\documentclass[11pt]
{article}

\newsavebox{\savepar}

\usepackage{amsmath}
\usepackage{amsfonts}
\usepackage{euscript}
\usepackage{oldgerm}
\usepackage{eufrak}
\usepackage{amsthm}
\usepackage{rotating}
\usepackage{calrsfs}
\DeclareMathAlphabet{\pazocal}{OMS}{zplm}{m}{n}
\makeindex
\newtheorem{corollary}{Corollary}
  \newtheorem{theorem}{Theorem}[section]
  \theoremstyle{definition}

  \theoremstyle{remark}

\newtheorem{lemma}[theorem]{\bf Lemma}
\newtheorem{proposition}[theorem]{\bf Proposition}

\theoremstyle{definition}

\theoremstyle{remark}

\newcommand{\R} {{\mathbb{R}}}

\newcommand{\N} {{\mathbb{N}}}

\begin{document}

\title{ Richardson extrapolation for the discrete iterated modified projection solution  }
\author{  Rekha P.  KULKARNI \thanks{Department of Mathematics, I.I.T. Bombay, 
Powai, Mumbai 400076, India,  rpk@math.iitb.ac.in,   } and Gobinda RAKSHIT \thanks{
gobindarakshit@math.iitb.ac.in} 
}

\date {}
\maketitle

\begin{abstract}

Approximate solutions of  Urysohn integral equations using projection methods  involve  integrals which need to be evaluated using a numerical quadrature formula. It gives rise to the discrete versions of the projection methods. For $r \geq 1,$ a space of piecewise polynomials of 
degree $\leq r - 1$ with respect to an uniform partition is chosen to be the approximating space and the projection 
is chosen to be the interpolatory projection at $r$ Gauss points. Asymptotic expansion for 
the iterated modified projection solution is available in literature. 
In this paper, we obtain an asymptotic  expansion for the  discrete  iterated modified projection solution and use 
Richardson extrapolation to improve the order of convergence. 
Our results indicate  a choice of a numerical quadrature which preserves 
the order of convergence in the continuous case. 
 Numerical results are given for a specific example.
\end{abstract}



\newpage
\setcounter{equation}{0}
\section {Introduction}

Let $\mathcal{X} = L^\infty [0, 1]$ and 
consider the following  Urysohn integral operator 
\begin{equation}\label{eq:1.1}
\mathcal{K} (x)(s)  = \int_0^1 \kappa (s, t, x (t))  d t, \;\;\; s \in [0, 1], \; x \in \mathcal{X},
\end{equation}
where $\kappa (s, t, u)$ is a continuous real valued function
defined on 
$ \Omega =[0, 1]\times[0, 1] \times \R.$
Then  $\mathcal{K} $ is a compact operator from $L^\infty [0, 1]$ to $C [0, 1].$ 
Assume that for $ f \in C[0, 1],$
\begin{equation}\label{eq:1.2}
x - \mathcal{K} (x) = f
\end{equation}
has a unique solution $\varphi.$

We are interested in  computable approximate solutions of the above equation. We consider some projection methods associated 
with a sequence of interpolatory projections converging to the Identity operator pointwise.

For $r \geq 1,$ let $\mathcal{X}_n$ denote the space of piecewise polynomials of degree $\leq r - 1$ with respect to 
a uniform partition of $[0, 1]$ with $n $ subintervals. Let $\displaystyle {h = \frac {1}{n}}$ denote the 
mesh of the partition and let $Q_n: C [0, 1] \rightarrow \mathcal{X}_n$ be
the interpolation operator  at $r$ Gauss points.  
In Grammont et al \cite{Gram3}  the following modified projection method is investigated:
\begin{equation}\nonumber
\varphi_n^M - \mathcal{K}_n^M (\phi_n^M) = f,
\end{equation}
where
\begin{equation}\label {eq:1.3}
\mathcal{K}_n^M (x) = Q_n \mathcal{K} (x) + \mathcal{K} (Q_n x) - Q_n \mathcal{K} (Q_n x).
\end{equation}
The iterated modified projection solution is defined as
\begin{equation}\nonumber
\tilde {\varphi}_n^M = \mathcal{K} (\phi_n^M) + f.
\end{equation}
If $\displaystyle { \frac {\partial \kappa} {\partial u} \in C^{3 r }(\Omega)}$ and $ f \in C^{2 r} [0, 1],$ then
the following orders of convergence are proved in 
 Grammont et al \cite{Gram3}:
\begin{equation}\nonumber
\| \varphi - \varphi_n^M \|_\infty= O (h^{3r}), \;\;\; \;\;\; \;\;\;  \| \varphi - \tilde{\varphi}_n^M \|_\infty= O (h^{4 r}).
\end{equation}
Under the assumption of $\displaystyle { \frac {\partial \kappa} {\partial u} \in C^{3 r +3}(\Omega)}$ and $ f \in C^{2 r + 2} [0, 1],$ 
the following asymptotic series expansion  is obtained in 
 Kulkarni-Nidhin \cite{Kul3} :
\begin{equation}\label{eq:A2}
  \tilde{\varphi}_n^M  =  \varphi  +   \zeta   \; h^{ 4 r } + O \left ( h^{ 4 r + 2} \right ),
\end{equation}
where the function $\zeta$ is independent of $h.$

In practice, it is necessary to replace all the integrals  by a numerical quadrature formula, giving rise to the
 discrete versions of the above methods.

 We choose a composite numerical quadrature formula with a degree of precision $d -1$ and with respect to a uniform partition of
$[0, 1]$ with $m = n p, \; p \in \N,$ subintervals. Let $\displaystyle {\tilde {h} = \frac {1}{m}}$ denote the mesh of this fine partition. Then the Nystr\"{o}m operator $\mathcal {K}_m$ is obtained by replacing the integral in the definition 
(\ref{eq:1.1}) of $\mathcal {K}$ and the discrete version of the operator $\mathcal{K}_n^M$ defined in (\ref{eq:1.3}) is obtained on replacing $\mathcal {K}$ by $\mathcal {K}_m$ and is denoted by $ \tilde  {\mathcal{K}}_n^M.$ The Nystr\"{o}m solution is denoted by $\varphi_m$  and it satisfies
\begin{eqnarray*}
\varphi_m - \mathcal {K}_m (\varphi_m) = f.
\end{eqnarray*}
The discrete modified projection  and the  discrete iterated modified projection solutions are denoted respectively by $z_n^M$ and $\tilde{z}_n^M$ and are defined as
\begin{eqnarray*}
z_n^M - \tilde {\mathcal{K}}_n^M (z_n^M) = f, \;\;\; \tilde{z}_n^M = \mathcal {K}_m z_n^M + f.
\end{eqnarray*}
In Kulkarni-Rakshit \cite{Kul1} the following estimates are proved:

If $\displaystyle {d \geq 2 r, \; \kappa \in C^d (\Omega), \;\;\; \frac {\partial \kappa } {\partial u} \in C^{2 r } (\Omega ) \mbox { and } \;  f \in C^{d}
 ( [0, 1]),}$ then
\begin{equation} \label{eq:1.7}
\|{z}_n^M - \varphi \|_\infty = O \left ( \max \left \{\tilde{h}^d, h^{3 r} \right \} \right).
\end{equation}
\hspace*{0.8 cm} If $\displaystyle {d \geq 2 r, \;\;\; \frac {\partial \kappa } {\partial u} \in 
C^{\max \{d, 3 r \} } (\Omega ) \mbox { and } \;  f \in C^{d}
 ( [0, 1]),}$ then
\begin{equation} \label{eq:1.8}
 \|\tilde{z}_n^M - \varphi \|_\infty = O \left  (\max \left \{\tilde{h}^d, h^{4 r} \right \} \right).
\end{equation}
In this paper we show that 
if $$\displaystyle {d \geq 2 r, \; \frac {\partial \kappa} {\partial u}  \in C^{ \max \{d, \; 3 r + 3 \}} (\Omega),
\ \;\;\; \mbox {and} \;\;\; f \in C^{\max \{d, \; 2 r + 2 \}} [0, 1]},$$
 then in fact
\begin{eqnarray} \label{eq:1.9}
 \tilde{z}_n^M = \varphi_m + \alpha \; h^{4 r} + O \left (  h^{2 r} \max \left \{ \tilde{h}^{d}, h^{2 r + 2} \right \} \right ), \end{eqnarray}
where the term $\alpha$ is independent of $h.$
 If we choose $\tilde{h}$ and $d$ such that $\tilde{h}^{d} \leq h^{2 r + 2},$ then using the Richardson extrapolation, an estimate of $\varphi_m$ of the order of $h^{4 r + 2}$ could be obtained.

The proof of  (\ref{eq:A2}) is based on an asymptotic expansion for the term 
$\mathcal {K}' (\varphi) (Q_n v - v),$ where $\mathcal {K}' (\varphi)$ denotes the Fr\'echet derivative of 
$\mathcal {K}$ at $\varphi$ and $v$ is a smooth function. This is obtained by using the Euler-MacLaurin formula. 
The main difficulty in proving (\ref{eq:1.9}) is that 
a discrete version of the Euler-MacLaurin formula is not available which is needed for obtaining an 
asymptotic expansion for $\mathcal {K}_m' (\varphi_m) (Q_n v - v).$ We  prove (\ref{eq:1.9}) using a different approach.

The paper is arranged as follows. In Section 2 we define discrete versions of the modified projection method and of its iterated version.
Section 3 is devoted to the asymptotic series expansions for $\mathcal {K}_m' (\varphi_m) (Q_n v - v)$ and for terms 
involving higher order Fr\'echet derivatives of the operator $\mathcal{K}_m$ at $\varphi_m.$ In Section 4 we prove our main result,
that is, the expansion (\ref{eq:1.9}). 
Numerical results are given in Section 5.

\setcounter{equation}{0}
\section{Discrete Projection Methods}

In the discrete projection methods, all the integrals are replaced by a numerical quadrature formula. We first 
define a numerical quadrature formula and the Nystr\"{o}m operator $ \mathcal{K}_m$ which approximates 
$ \mathcal{K}.$
  Let $m \in \mathbb{N}$ and
 consider the following uniform partition of $[0, 1]:$
\begin{equation}\label{eq:2.1}
\Delta_m: 0  <  \frac{1} {m}  < \cdots <   \frac{m-1} {m}   <  1.
\end{equation}
Let 
\begin{equation}\nonumber
 \tilde{h} = \frac {1} {m} \;\;\; \mbox {and} \;\;\; s_i = \frac {i} {m},  \;\;\; i = 0, \ldots, m.
 \end{equation}
Consider a basic quadrature rule
  \begin{equation}\nonumber
  \int_0^1 f (t) d t \approx \sum_{i=1}^\rho \omega_i f (\mu_i)
  \end{equation}
  and let
\begin{equation}\nonumber
\zeta_i^j  = s_{j-1} +   \mu_i \;  \tilde h,  \;\;\; i = 1, \ldots, \rho, \;\;\; j = 1, \ldots, m.
\end{equation}
A composite integration rule with respect to the partition  (\ref{eq:2.1}) 
is then defined as
\begin{eqnarray}\label{eq:2.3}
 \int_0^1 f (t) d t &=& \sum_{j=1}^m \int_{s_{j-1}}^{s_j} f (t) d t 
  \approx  \tilde h \sum_{j=1}^m  \sum_{i=1}^\rho \omega_i \; f (\zeta_i^j ).
\end{eqnarray}
For $k \geq 1,$ let
$$C_{\Delta_m}^k \left ([0, 1] \right )= \left \{g \in L^\infty[0, 1] : g|_{\left [s_{j-1}, s_j \right]}  \in C^k \left (\left [s_{j-1}, s_j \right ] \right), \; j = 1,\ldots, m \right \}.$$
The error in the numerical quadrature  is assumed to be of the following form: 
There is a positive integer  $ d $ such that    if $f \in C_{\Delta_m}^{d} ([0, 1]),$
then
\begin{equation}\label{eq:2.4}
\left | \int_0^1 f (t) d t 
 - \tilde h \sum_{j=1}^m  \sum_{i=1}^\rho \omega_i \; f (\zeta_i^j ) \right | \leq C_1\left \|f^{(d)} \right \|_\infty\tilde {h}^{d},
\end{equation}
where  $f^{(d)} $ denotes $d$ th derivative of $f$ and $C_1$ is a constant independent of $\tilde{h}.$

Assume that $\kappa \in C^d (\Omega)$ and that $f \in C^{d} ( [0, 1]).$ Then it follows that
  $\varphi \in C^{d}( [0, 1]). $ We also assume throughout that
  $\displaystyle{\left(I - \mathcal{K}' (\varphi)\right)^{-1}}$  is a bounded linear operator from $C[0, 1]$ to itself.

 We replace the integral in (\ref{eq:1.1}) by the numerical quadrature formula (\ref{eq:2.3}) and define
the Nystr\"{o}m operator as
\begin{equation}\nonumber
\mathcal{K}_m (x) (s)  =  \tilde {h}  \sum_{j=1}^m  \sum_{i=1}^\rho \omega_i \;  \kappa {\left(s,  \zeta_i^j, x{\left(\zeta_i^j \right)} \right)}, \;\;\; s \in [0, 1], \;\;\; x \in C ([0, 1]).
\end{equation}
Then since $\kappa \in C^d (\Omega)$  and $\varphi \in C^{d}( [0, 1]), $ it follows that
\begin{eqnarray}\label{eq:2.5}
\|\mathcal{K} (\varphi)  - \mathcal{K}_m (\varphi ) \|_\infty  = O \left (\tilde{h}^{d} \right).
\end{eqnarray}
In the Nystr\"{o}m method,  (\ref{eq:1.2}) is approximated by
\begin{equation}\label{eq:2.6}
x_m - \mathcal{K}_m (x_m) = f.
\end{equation}
Fix $\delta > 0$ and define
 \begin{equation}\nonumber
B (\varphi, \delta ) = \{ x: \|x - \varphi \|_\infty \leq \delta \}.
\end{equation}  
Then for $ m $ sufficiently large, the  equation (\ref{eq:2.6})  has a unique solution $\varphi_m$ in $B (\varphi, \delta )$
and
\begin{eqnarray}\label{eq:2.8}
\|\varphi - \varphi_m \|_\infty  = O \left (\tilde{h}^{d} \right).
\end{eqnarray}
Note that the Fr\'echet derivative of $\mathcal{K}_m$ is given by
\begin{equation*}
\mathcal{K}_m' (x) v (s)  =  \tilde {h}  \sum_{j=1}^m  \sum_{i=1}^\rho \omega_i \;  \frac {\partial \kappa } {\partial u} \left (s,  \zeta_i^j, x \left (\zeta_i^j \right ) \right) v \left (\zeta_i^j \right ), \;\;\; s \in [0, 1], \;\;\; v \in C ([0, 1]).
\end{equation*}
If $\displaystyle { \frac {\partial \kappa} {\partial u} \in C^d (\Omega)} $ and $v \in C_{\Delta_m}^d ([0, 1]),$ then from (\ref{eq:2.4}),
\begin{eqnarray}\label{eq:2.9}
\|\mathcal{K}' (\varphi) v - \mathcal{K}_m' (\varphi ) v \|_\infty \leq C_2 \|v\|_{d, \infty} \tilde{h}^d,
\end{eqnarray}
where $\displaystyle {\|v\|_{d, \infty} = \max_{0 \leq j \leq d} \|v^{(j)}\|_\infty}$ and $C_2$ is a constant 
independent of $\tilde {h}.$

We now define  the interpolatory projection at $r$ Gauss points.
Let $n \in \mathbb{N}$ and consider the following uniform partition of $[0, 1]:$
\begin{equation}\nonumber
\Delta_n: 0  <  \frac{1} {n}  < \cdots <   \frac{n-1} {n}   <  1.
\end{equation}
Define
\begin{equation}\nonumber
t_j = \frac {j} {n}, \;\;\; j = 0, \ldots, n \;\;\; \mbox{and} \;\;\; h = t_{j} - t_{j-1} = \frac {1} {n}.
\end{equation}
 For a positive integer $r,$
 let 
 \begin{equation}\nonumber
  \mathcal{X}_n = \left \{  g \in L^\infty [0, 1]: g |_{\left [t_{j-1}, t_j \right ]} \; \mbox{is a polynomial of degree} \; \leq r - 1, \; j = 1, \ldots, n
  \right \}.
  \end{equation}
  Let $ \{q_1,  \cdots,   q_r \}$
  denote the Gauss-Legendre zeros of order $r$ in $[0, 1].$ 
  The  $n r $ collocation nodes are chosen as follows:
   \begin{equation}\nonumber
  \tau_i^k  = t_{k-1} + q_i \; h, \;\;\; i = 1, \ldots, r, \;\;\; k = 1, \ldots, n.
  \end{equation}
  Define the interpolation operator $Q_n: C [0, 1] \rightarrow \mathcal{X}_n$ as
   \begin{equation}\nonumber
   (Q_n x) (\tau_i^k) = x (\tau_i^k), \;\;\; i = 1, \ldots, r, \;\;\; k = 1, \ldots, n.
   \end{equation}
Using the Hahn-Banach extension theorem as in Atkinson et al \cite {AtkG}, $Q_n$ can be extended to $L^\infty [0, 1].$  
Note that  for $x \in C [0, 1],$ $Q_n x \rightarrow x$ as $ n \rightarrow \infty.$ 
As a consequence, 
\begin{equation}\label {eq:2.10}
\sup_{n} \|Q_n |_{C [0, 1]}\| \leq C_3.
\end{equation}
Also, if $u \in C^r ([0,1]),$
then
\begin{eqnarray} \label{eq:2.14}
\left \|u - Q_n u \right \|_\infty \leq C_4 \left \|u^{(r)} \right \|_\infty h^r,
\end{eqnarray}
where  $C_4$ is a constant independent of $n.$
Throughout this paper we choose $d \geq 2 r$ and
$  m = p n $ {for some} $ p \in \N.$ 
Then $ \displaystyle { \tilde {h} = \frac { h} {p} \leq h.}$
It follows from (\ref{eq:2.8}), (\ref{eq:2.10}) and (\ref{eq:2.14}) that
\begin{eqnarray} \nonumber
\|\varphi_m - Q_n \varphi_m\|_\infty \leq \|\varphi - Q_n \varphi\|_\infty + \|(I - Q_n ) (\varphi_m - \varphi)\|_\infty
= O \left (\max \left  \{h^r, \tilde{h}^d  \right \} \right).
\end{eqnarray}
 Hence it follows that
\begin{eqnarray} \label{eq:2.15}
\|\varphi_m - Q_n \varphi_m\|_\infty 
= O \left(h^r \right).
\end{eqnarray}
 The discrete modified projection operator is defined by replacing ${\mathcal{K}}$ with ${\mathcal{K}}_m$ in (\ref{eq:1.3}) as follows:
\begin{equation}\label{eq:2.16}
\tilde{\mathcal{K}}_n^M (x) = Q_n {\mathcal{K}}_m (x) + {\mathcal{K}}_m (Q_n x) - Q_n  {\mathcal{K}}_m (Q_n x).
\end{equation}
Discrete Modified Projection Method: For $n$ and $m$ big enough, 
\begin{equation}\label{eq:2.17}
 x_n -  \tilde{\mathcal{K}}_n^M (x_n) =  f
\end{equation}
has a unique solution ${z}_n^M$ in a $B (\varphi, \delta).$

\noindent
The Discrete Iterated  Modified Projection solution is defined as
\begin{equation}\label{eq:2.18}
 \tilde{z}_n^M =  {\mathcal{K}}_m ( {z}_n^M) +  f.
\end{equation}


\setcounter{equation}{0}
\section{Asymptotic Series Expansions}


By assumption, $I - \mathcal {K}' (\varphi) $ is invertible.
Let
$$M = \left (I - \mathcal {K}' (\varphi) \right )^{-1} \mathcal {K}' (\varphi), \;\;\; \Psi (t) = (t - q_1) \cdots (t - q_r), \; t \in [0, 1].$$
We quote the following asymptotic series expansions from Kulkarni-Nidhin \cite{Kul3}.

\noindent
If $ \frac{\partial \kappa}{\partial u}\in C^{2 r + 2}(\Omega)$ and $\psi \in C^{2 r + 2} ([0, 1]),$ then
\begin{equation} \label {eq:3.1}
 \mathcal {K}' (\varphi)  ( Q_n \psi - \psi)   =  
 {T} ( \psi ) h^{ 2 r } + O (h^{2 r + 2}),
\end{equation}
\begin{equation} \label {eq:3.2}
  M  ( Q_n \psi - \psi)   =  
 {U} ( \psi ) h^{ 2 r } + O (h^{2 r + 2}),
\end{equation}
where for $ s \in [0, 1],$
\begin{equation*} 
  T ( \psi )  (s) = d_{ 2 r, 2 r } \;   \mathcal {K}' (\varphi) 
 { \psi^{( 2 r )}} (s)+ \sum_{i = r }^{2 r - 1} d_{2 r , i} \left [ \left ( \frac {\partial} {\partial t} \right )^{2 r - 1 - i} \ell (s, t) \psi^{(i)} (t) 
\right ]_{t = 0}^1,
\end{equation*}
\begin{equation*} 
  U ( \psi )  (s) = d_{ 2 r, 2 r } \;   M
 { \psi^{( 2 r )}} (s)+ \sum_{i = r }^{2 r - 1} d_{2 r , i} \left [ \left ( \frac {\partial} {\partial t} \right )^{2 r - 1 - i} m (s, t) \psi^{(i)} (t) 
\right ]_{t = 0}^1,
\end{equation*}
$$ d_{ 2 r , i} = - \int_0^1  \Phi_{i} (\tau) \frac {B_{2 r  - i}  (\tau) } { ( 2 r  - i)! } \;
\Psi (\tau) d \tau, \;\;\; i = r , \ldots, 2 r,$$
with
\begin{equation}\nonumber
 \Phi_{i} (\tau) = \int_0^1 \frac {(\sigma - \tau)^{i - r }} { ( i - r )!} \; 
\frac {[q_1, q_2, \cdots, q_r, \tau ] (\cdot - \sigma)_+^{r-1} } { (r-1)!} 
d \sigma. 
\end{equation}
Note that the coefficients $d_{2 r, i}$ are independant of $h.$
Also,
\begin{equation*} 
 \mathcal {K}'' (\varphi)  ( Q_n \psi - \psi)^2   =  
 {V_1} ( \psi ) h^{ 2 r } + O (h^{2 r + 2}), \;\;\; \mathcal {K}^{(3)} (\varphi)  ( Q_n \psi - \psi)^3   =  
 {V_2} ( \psi ) h^{ 3 r } + O (h^{3 r + 1}),
\end{equation*}
where
\begin{eqnarray*}
{V_1} ( \psi ) = \left ( \int_0^1 \Psi (\tau)^2 \Phi_{r}(\tau)^2 d \tau \right ) \mathcal {K}'' (\varphi)   
\left (\psi^{(r)} \right )^2
\end{eqnarray*}
and
\begin{eqnarray*}
{V_2} ( \psi ) = \left ( \int_0^1 \Psi (\tau)^3 \Phi_{r}(\tau)^3 d \tau \right ) \mathcal {K}^{(3)} (\varphi)   
\left (\psi^{(r)} \right )^3.
\end{eqnarray*}
Let $\displaystyle {\frac {\partial^4 \kappa} {\partial u^4} \in C^{2 r} (\Omega) }$ and $v_k \in C ([0, 1]), \; 1 \leq k \leq 4.$ The 
Fr\'echet derivatives of $\mathcal{K}_m$ are as follows:
\begin{equation}\nonumber
 \mathcal {K}_m^{(k)} (x) (v_1, \ldots, v_k )  (s) = \tilde {h}  \sum_{j=1}^m  \sum_{i=1}^\rho  \omega_i \;  \frac {\partial^k \kappa } {\partial u^k} \left(s,  \zeta_i^j, x \left (\zeta_i^j \right )\right) v_1 \left (\zeta_i^j \right )  \cdots v_k \left (\zeta_i^j \right ), \;\;\; s \in [0, 1].
\end{equation}
We introduce the following notations:
$$ D^{(i, j, k)} \kappa  (s, t, u) =  \frac {\partial ^{i + j + k} \kappa } {\partial s^i \partial t^j \partial u^k} (s, t, u),
\;\;\; C_{5} = \max_{\stackrel{0 \leq i + j \leq 2 r}{1 \leq k \leq 4}}  \max_{\stackrel {s, t \in [0, 1]}{|u| \leq \|\varphi \|_\infty + \delta }}
\left | D^{(i, j, k)} \kappa  (s, t, u) \right |.$$
Using the Mean Value Theorem, it can be seen that 
\begin{eqnarray}\label{eq:A1}
\left \|\mathcal {K}_m'  (\varphi) - \mathcal {K}_m'  (\varphi_m) \right \| \leq \left (\sum_{i=1}^\rho |\omega_i| \right ) C_5 \tilde{h}^d.
\end{eqnarray}

We now prove asymptotic series expansions for the first three  Fr\'echet derivatives of the Nystr\"{o}m operator $\mathcal{K}_m$ at
$\varphi_m.$
\begin{proposition}  \label {prop:3.1}
If $ \frac{\partial \kappa}{\partial u}\in C^{ \max \{d, 2 r + 2 \}}(\Omega)$ and $f \in C^{ \max \{d, 2 r + 2 \}} ( [0, 1]), $   then
\begin{eqnarray}\label {eq:3.3}
 \mathcal {K}_m' (\varphi_m)  ( Q_n \varphi_m - \varphi_m)  
   & = & {T} ( \varphi ) h^{ 2 r } + O \left ( \max \left \{\tilde{h}^d,  h^{2 r + 2} \right \} \right ),
\end{eqnarray}
\begin{eqnarray} \label {eq:3.4}
 \mathcal {K}_m'' (\varphi_m)  ( Q_n \varphi_m - \varphi_m)^2  
   & = & {V_1} ( \varphi ) h^{ 2 r }  + O \left ( \max \left \{\tilde{h}^d,  h^{2 r + 2} \right \} \right ),
\end{eqnarray}
\begin{eqnarray} \label {eq:3.5}
 \mathcal {K}_m^{(3)}(\varphi_m)  ( Q_n \varphi_m - \varphi_m)^3  
   & = & {V_2} ( \varphi ) h^{ 3 r }  + O \left ( \max \left \{\tilde{h}^d,  h^{3 r + 1} \right \} \right ).
\end{eqnarray}
\end{proposition}
\begin{proof}
We write
\begin{eqnarray*}
\mathcal {K}_m'  (\varphi) ( Q_n \varphi - \varphi)  = \mathcal {K}'  (\varphi) ( Q_n \varphi - \varphi) +
(\mathcal {K}_m'  (\varphi) - \mathcal {K}'  (\varphi)) ( Q_n \varphi - \varphi).
\end{eqnarray*}
Since $d \geq 2 r$ and $ \tilde {h} \leq h,$ from the estimate (\ref{eq:2.9}) we obtain,
$$ \|(\mathcal {K}_m'  (\varphi) - \mathcal {K}'  (\varphi)) ( Q_n \varphi - \varphi) \|_\infty 
\leq C_2 \left \| Q_n \varphi - \varphi \right \|_{d, \infty} \tilde{h}^d = C_2 \left \|  \varphi \right \|_{d, \infty} \tilde{h}^d.$$
Since $\varphi \in C^{ 2 r + 2} [0, 1],$ 
 it  follows from (\ref {eq:3.1}) that
\begin{equation*} 
 \mathcal {K}_m' (\varphi)  ( Q_n \varphi - \varphi)   =  
 {T} ( \varphi ) h^{ 2 r } + O (h^{2 r + 2}).
\end{equation*}
Since by (\ref{eq:2.8}),
$\|\varphi - \varphi_m \|_\infty 
= O (\tilde{h}^d),$
we obtain
\begin{eqnarray*}
 \mathcal {K}_m' (\varphi)  ( Q_n \varphi_m - \varphi_m)  &=&
  \mathcal {K}_m' (\varphi)  ( Q_n \varphi - \varphi)  + \mathcal {K}_m' (\varphi) (Q_n - I) (\varphi_m - \varphi)\\
  & = & {T} ( \varphi ) h^{ 2 r } + O \left ( \max \left \{\tilde{h}^d,  h^{2 r + 2} \right \} \right ).
  \end{eqnarray*}
  Using the estimate (\ref{eq:A1}), we thus obtain
\begin{eqnarray*}
 \mathcal {K}_m' (\varphi_m)  ( Q_n \varphi_m - \varphi_m)  &=&  \mathcal {K}_m' (\varphi)  ( Q_n \varphi_m - \varphi_m) + [\mathcal {K}_m'  (\varphi_m) - \mathcal {K}_m'  (\varphi)] ( Q_n \varphi_m - \varphi_m )\\
   & = & {T} ( \varphi ) h^{ 2 r } + O \left ( \max \left \{\tilde{h}^d,  h^{2 r + 2} \right \} \right ).
\end{eqnarray*}
This completes the proof of (\ref {eq:3.3}). The proofs of (\ref {eq:3.4}) and (\ref {eq:3.5}) are similar.
\end{proof}
\setcounter{equation}{0}
\section{Discrete Iterated Modified Projection method }

In this section we prove our main result about the asymptotic series expansion for
the discrete iterated modified projection solution $\tilde {z}_n^M.$

Using the generalised Taylor series expansion, we obtain 
\begin{eqnarray*}\nonumber
\tilde{z}_n^M - \varphi_m &= &  \mathcal {K}_m (z_n^M) - \mathcal {K}_m (\varphi_m)
= \mathcal {K}_m' (\varphi_m) (z_n^M - \varphi_m) + R ({z}_n^M - \varphi_m),
\end{eqnarray*}
where
 \begin{eqnarray}\nonumber
&&R  ({z}_n^M - \varphi_m)  (s)
 = \int_0^1  {(1 - \theta) }   \mathcal {K}_m^{''} \left (\varphi_m + \theta  
 ({z}_n^M - \varphi_m) \right ) ({z}_n^M - \varphi_m)^2  (s)    d \theta.
\end{eqnarray}
It follows that
\begin{eqnarray*}
\|R  ({z}_n^M - \varphi_m) \|_\infty \leq \frac {C_{5}} {2}  \left ( \sum_{i=1}^\rho | \omega_i | \right ) \|z_n^M - \varphi_m \|_\infty^2. 
\end{eqnarray*}
If $\displaystyle {d \geq 2 r, \; \kappa \in C^d (\Omega), \; \frac {\partial \kappa } {\partial u} \in C^{2 r } (\Omega ) \mbox { and } \;  f \in C^{d}
 ( [0, 1]),}$ then
using (\ref {eq:1.7}) we deduce that
\begin{eqnarray}\label {eq:4.1}
\tilde{z}_n^M - \varphi_m 
& = & \mathcal {K}_m' (\varphi_m) (z_n^M - \varphi_m)  + O \left (\max \{\tilde{h}^{d}, h^{3 r} \}^2 \right ).
\end{eqnarray}
Our aim is to obtain an asymptotic series expansion for the first term on the right hand side of the above equation.

We first prove the following  preliminary results which are needed later on.


\begin{lemma}\label{lemma:4.1}
If $\displaystyle {\frac {\partial \kappa} {\partial u}  \in C^{ 3 r + 3} (\Omega)},$ 
    then for $x \in B (\varphi, \delta)$  and for $1 \leq k \leq 4,$
  \begin{eqnarray} \label {eq:4.2}
&&\left \|  (I - Q_n) \mathcal {K}_m^{(k)} (x)  \right \| 
   = O \left ( h^{ r} \right),
\end{eqnarray}
\begin{eqnarray} \label {eq:4.3}
&& \left \| \mathcal {K}_m'  (\varphi_m) (I - Q_n) \mathcal {K}_m^{(k)} (x)  \right \| 
   = O \left ( h^{2 r} \right)
\end{eqnarray}
and
\begin{eqnarray} \label {eq:4.4}
&&\left \|  \mathcal {K}_m'  (\varphi_m) (I - Q_n )  \mathcal {K}_m'  (\varphi_m) (I - Q_n)  \mathcal {K}_m'  (x)  \right \|
= O (h^{4 r}).
\end{eqnarray}
\end{lemma}
\begin{proof}
Note that  for $1 \leq k \leq 4$ and for $v_1, \ldots, v_k \in C ([0, 1]),$
\begin{eqnarray*}
\mathcal {K}_m^{(k)} (x) (v_1, \ldots, v_k )  (s) &=& \tilde {h}  \sum_{j=1}^m  \sum_{i=1}^\rho  \omega_i \;  \frac {\partial^{k} \kappa } { \partial u^k} \left(s,  \zeta_i^j, x (\zeta_i^j)\right) v_1 (\zeta_i^j)  \cdots v_k (\zeta_i^j)
\end{eqnarray*}
and for $p = 1, \ldots, 2 r,$
\begin{eqnarray*}
 \left (\mathcal {K}_m^{(k)} (x) (v_1, \ldots, v_k ) \right)^{(p)} (s) &=& \tilde {h}  \sum_{j=1}^m  \sum_{i=1}^\rho  \omega_i \;  \frac {\partial^{p+k} \kappa } {\partial s^p \partial u^k} \left(s,  \zeta_i^j, x (\zeta_i^j)\right) v_1 (\zeta_i^j)  \cdots v_k (\zeta_i^j).
\end{eqnarray*}
Hence, for $x \in B (\varphi, \delta),$
\begin{eqnarray*}
 \left \|\left (\mathcal {K}_m^{(k)} (x) (v_1, \ldots, v_k ) \right)^{(p)} \right \|_\infty 
  & \leq & C_5 \left (  \sum_{i=1}^\rho  |\omega_i| \right )  \|v_1\|_\infty \ldots \|v_k\|_\infty.
\end{eqnarray*}
From the estimate (\ref{eq:2.14}),
\begin{eqnarray*}
\|(I - Q_n) \mathcal {K}_m^{(k)} (x) (v_1, \ldots, v_k ) \|_\infty &\leq& C_4  \left \|\left (\mathcal {K}_m^{(k)} (x) (v_1, \ldots, v_k ) \right)^{(r)} \right \|_\infty h^r\\
&\leq & C_4 C_5 \left (  \sum_{i=1}^\rho  |\omega_i| \right )  \|v_1\|_\infty \ldots \|v_k\|_\infty h^r.
\end{eqnarray*}
By taking the supremum over the set $\{ (v_1, \ldots, v_k): \|v_1\|_\infty \leq 1, \ldots, \|v_k\|_\infty \leq 1\},$ we obtain (\ref{eq:4.2}).

We recall the following result from Kulkarni-Rakshit \cite{Kul1}: If $v \in C^{2 r } [0, 1],$ then
\begin{eqnarray} \nonumber
\| \mathcal {K}_m'  (\varphi_m) (I - Q_n) v \|_\infty 
& \leq & C_6  \|\ell_m \|_{r, \infty} \|v \|_{2 r, \infty } \; h^{2 r},
\end{eqnarray}
where
$$\displaystyle {\ell_m (s, t) = \frac {\partial \kappa} {\partial u} (s, t, \varphi_m (t)), \; s, t \in [0, 1],
\;\;\; \|\ell_m \|_{r, \infty} = \max_{0 \leq i + j \leq r} \left \|D^{(i, j)} \ell_m \right \|_\infty }$$
and
$\displaystyle {C_6 =  \frac {1} {r!} 2^r   \| \Psi \|_\infty \left (\sum_{i=1}^\rho |\omega_i| \right )}$
is a constant independent of $n.$ 
Since $ \left \| \ell_m \right \|_{r, \infty} \leq C_5,$ it follows that 
\begin{eqnarray*}
\| \mathcal {K}_m'  (\varphi_m) (I - Q_n) \mathcal {K}_m^{(k)} (x) (v_1, \ldots, v_k ) \|_\infty 
&\leq &  (C_5)^{2}  C_6 \left (  \sum_{i=1}^\rho  |\omega_i| \right )    \|v_1\|_\infty \ldots \|v_k\|_\infty \; h^{2 r}.
\end{eqnarray*}
By taking the supremum over the set $\{ (v_1, \ldots, v_k): \|v_1\|_\infty \leq 1, \ldots, \|v_k\|_\infty \leq 1\},$ we obtain 
(\ref{eq:4.3}).

In order to prove (\ref{eq:4.4}),
we recall the following estimate from Kulkarni-Rakshit \cite{Kul1}:
If $u \in C^{2 r} [0, 1],$ then
\begin{eqnarray} \label {eq:4.6}
\|  \mathcal {K}_m'  (\varphi_m) (I - Q_n )  \mathcal {K}_m'  (\varphi_m) (I - Q_n) u \|_\infty \leq  ( C_6)^2 \; \| \ell_m \|_{ r, \infty}  \|\ell_m \|_{3 r, \infty}  \|u\|_{2 r, \infty} h^{4 r}.
\end{eqnarray}
Then for $v \in C [0, 1],$
\begin{eqnarray*}
&&\|  \mathcal {K}_m'  (\varphi_m) (I - Q_n )  \mathcal {K}_m'  (\varphi_m) (I - Q_n)  \mathcal {K}_m'  (x) v \|_\infty \\
& & \hspace*{2 cm} \leq
( C_6)^2 \; \| \ell_m \|_{ r, \infty}  \|\ell_m \|_{3 r, \infty}  \| \mathcal {K}_m'  (x) v\|_{2 r, \infty} h^{4 r}
\\
& & \hspace*{2 cm} \leq   (C_5)^2 (C_6)^2 \left (  \sum_{i=1}^\rho  |\omega_i| \right )   \;  \|\ell_m \|_{3 r, \infty}  \|v\|_\infty h^{4 r}.
\end{eqnarray*}
The estimate  (\ref{eq:4.4}) then follows by taking the supremum over the set $\{v \in C[0, 1]: \|v\|_\infty \leq 1 \}.$

\end{proof}


From Proposition 4.2 of Kulkarni-Rakshit \cite{Kul1} we recall that for all $m$ big enough, \\$I - {\mathcal{K}}_m'   (\varphi_m)$ is invertible 
and that $\displaystyle {\left \|\left (I - {\mathcal{K}}_m'   (\varphi_m)  \right)^{-1} \right \| \leq C_7,}$
where $C_7$ is a constant independent of $m.$
It can be easily checked that 
\begin{eqnarray*}
z_n^M - \varphi_m 
& = & - \left [ I - {\mathcal{K}}_m'   (\varphi_m) \right]^{-1}  \left \{  \mathcal {K}_m (\varphi_m) -  
\mathcal{K}_m'   (\varphi_m) \varphi_m - \tilde{\mathcal{K}}_n^M  ( z_n^M) +   \mathcal{K}_m'   (\varphi_m) z_n^M \right \}.
\end{eqnarray*}
Let
$ L_m =  \left ( I - {\mathcal{K}}_m'   (\varphi_m) \right)^{-1}.$
We write
\begin{eqnarray}\nonumber
&&\mathcal{K}_m'   (\varphi_m) ( z_n^M - \varphi_m ) \\\nonumber
&  &  \hspace*{1.5 cm} =
 - L_m  \mathcal{K}_m'   (\varphi_m) \left \{  \mathcal {K}_m (\varphi_m) -  \tilde{\mathcal{K}}_n^M (\varphi_m) \right \}\\\nonumber
&  &  \hspace*{1.5 cm} \;\;+ L_m \mathcal{K}_m'   (\varphi_m) \left \{
\tilde{\mathcal{K}}_n^M (z_n^M) - \tilde{\mathcal{K}}_n^M (\varphi_m)  - \left (  \tilde{\mathcal{K}}_n^M \right )' (\varphi_m) 
(z_n^M - \varphi_m) \right \}\\  \label {eq:4.7}
&  &  \hspace*{1.5 cm}  \;\;+ L_m \mathcal{K}_m'   (\varphi_m) \left \{
\left (\left (  \tilde{\mathcal{K}}_n^M \right )' (\varphi_m)  -  \mathcal{K}_m'   (\varphi_m) \right  ) (z_n^M - \varphi_m) \right \}.
\end{eqnarray}
 We obtain an asymptotic expansion for the first term on the right hand side of 
(\ref{eq:4.7}) and show that the second and the third terms are of the order of $\displaystyle {h^{2 r} \max \left \{ \tilde{h}^d, h^{2 r + 2} \right \} }.$
The following two lemmas are needed to obtain the results for the first term.

From now onwards we assume that 
$$d \geq 2 r, \; \frac {\partial \kappa} {\partial u}  \in C^{ \max \{d, 3 r + 3 \}} (\Omega)
\;\;\; \mbox{and} \;\;\; f \in C^{\max \{d, 2 r + 2 \}} [0, 1].$$

\begin{lemma}\label{lemma:4.2}
If  $d \geq 2 r, \; \frac {\partial \kappa} {\partial u}  \in C^{ \max \{d, 3 r + 3 \}} (\Omega)
\;\;\; \mbox{and} \;\;\; f \in C^{\max \{d, 2 r + 2 \}} [0, 1],$ then
\begin{eqnarray} \nonumber
&& L_m  \mathcal{K}_m'   (\varphi_m) (I - Q_n) \mathcal {K}_m'  (\varphi_m) ( Q_n - I) \varphi_m = 
 U (T (\varphi)) h^{4 r} + O \left ( h^{2 r} \max \left \{ \tilde{h}^d, h^{2 r + 2} \right \} \right )\\\label {eq:4.8}
\end{eqnarray}
and
\begin{eqnarray}\nonumber
&&L_m  \mathcal{K}_m'  (\varphi_m) (I - Q_n) \mathcal {K}_m'' (\varphi_m) ( ( Q_n - I) \varphi_m )^2 = 
 U (V_1 (\varphi)) h^{4 r} +O \left ( h^{2 r} \max \left \{ \tilde{h}^d, h^{2 r + 2} \right \} \right ).\\ \label {eq:4.9}
\end{eqnarray}
\end{lemma}
\begin{proof}
Using the resolvent identity, we write
\begin{eqnarray*}
   L_m  &= & \left [ I - {\mathcal{K}}'   (\varphi) \right]^{-1}
+   \left [ I - {\mathcal{K}}'   (\varphi) \right]^{-1} \left [{\mathcal{K}}_m'   (\varphi_m) - 
 {\mathcal{K}_m}'   (\varphi) \right ] L_m
 \\
 &&  + ~ \left [ I - {\mathcal{K}}'   (\varphi) \right]^{-1}   \left [{\mathcal{K}}_m'   (\varphi) - 
 {\mathcal{K}}'   (\varphi) \right ] L_m.
\end{eqnarray*}
Let
\begin{eqnarray*}
y_n & = &  \mathcal{K}_m'   (\varphi_m) (I - Q_n) \mathcal {K}_m'  (\varphi_m) ( Q_n - I) \varphi_m \\&=&\mathcal{K}_m'   (\varphi_m) (I - Q_n) \mathcal {K}_m'  (\varphi_m) (Q_n - I) \varphi\\
&& + ~ \mathcal{K}_m'   (\varphi_m) (I - Q_n) \mathcal {K}_m'  (\varphi_m) (Q_n - I) (\varphi_m - \varphi ).
\end{eqnarray*}
From (\ref {eq:2.8}), (\ref {eq:4.3}) and (\ref {eq:4.6}),  it follows that
\begin{eqnarray}\label{eq:*}
\|y_n\|_\infty = O \left (  h^{2 r} \max \left \{ \tilde{h}^{d}, h^{2 r} \right \} \right ) =  O \left (  h^{4 r}  \right ).
\end{eqnarray}
Note that
\begin{eqnarray}\nonumber
L_m y_n 
& = & \left [ I - {\mathcal{K}}'   (\varphi) \right]^{-1}  y_n \\\nonumber
&& + \left [ I - {\mathcal{K}}'   (\varphi) \right]^{-1}   \left [{\mathcal{K}}_m'   (\varphi_m) - 
 {\mathcal{K}_m}'   (\varphi) \right ]
L_m  y_n\\\label {eq:4.10}
&& + \left [ I - {\mathcal{K}}'   (\varphi) \right]^{-1}  \left [{\mathcal{K}}_m'   (\varphi) - 
 {\mathcal{K}}'   (\varphi) \right ] L_m
 y_n.
\end{eqnarray}
Consider the first term
\begin{eqnarray*}
\left [ I - {\mathcal{K}}'   (\varphi) \right]^{-1}  y_n 
& = & \left [ I - {\mathcal{K}}'   (\varphi) \right]^{-1}  
\mathcal{K}'   (\varphi) (I - Q_n) \mathcal {K}_m'  (\varphi_m) (Q_n \varphi_m - \varphi_m )\\
& + & \left [ I - {\mathcal{K}}'   (\varphi) \right]^{-1}  
\left ( \mathcal {K}_m'  (\varphi_m) - \mathcal{K}'   (\varphi) \right ) (I - Q_n)  \mathcal{K}_m'   (\varphi_m)  (Q_n \varphi_m - \varphi_m )
\end{eqnarray*}
Note that from (\ref{eq:3.2}) and (\ref {eq:3.3}),
\begin{eqnarray*}
&&\left [ I - {\mathcal{K}}'   (\varphi) \right]^{-1}  
\mathcal{K}'   (\varphi) (I - Q_n) \mathcal {K}_m'  (\varphi_m) (Q_n \varphi_m - \varphi_m ) \\
&& \hspace*{4 cm} = U (T (\varphi)) h^{4 r} + O \left ( h^{2 r} \max \left \{ \tilde{h}^d, h^{2 r + 2} \right \} \right ).
\end{eqnarray*}
On the other hand,
\begin{eqnarray*}
(I - Q_n)  \mathcal{K}_m'   (\varphi_m)  (Q_n \varphi_m - \varphi_m )
&= &(I - Q_n)  \mathcal{K}_m'   (\varphi)  (Q_n \varphi - \varphi ) \\
&+& (I - Q_n)  \mathcal{K}_m'   (\varphi)  (Q_n - I) (\varphi_m - \varphi)\\
&+& (I - Q_n)  (\mathcal{K}_m'   (\varphi_m)  -  \mathcal{K}_m'   (\varphi) )(Q_n \varphi_m - \varphi_m ).
\end{eqnarray*}
From Kulkarni-Rakshit \cite[Proposition 2.3]{Kul1}, we have 
\begin{eqnarray*}
\|(I - Q_n)  \mathcal{K}_m'   (\varphi)  (Q_n \varphi - \varphi ) \|_\infty &=& O \left (h^{3 r} \right ).
\end{eqnarray*}
Combining the above estimate with the estimates (\ref{eq:2.8}),  (\ref{eq:2.15}), (\ref{eq:A1}) and (\ref{eq:4.2}), we obtain
\begin{eqnarray*}
&& \left \|\left [ I - {\mathcal{K}}'   (\varphi) \right]^{-1}  
\left ( \mathcal {K}_m'  (\varphi_m) - \mathcal{K}'   (\varphi) \right ) (I - Q_n)  \mathcal{K}_m'   (\varphi_m)  (Q_n \varphi_m - \varphi_m ) \right \|_\infty  = 
O \left (\tilde{h}^d  h^{3 r}  \right ). 
\end{eqnarray*}
Thus,
\begin{eqnarray*}
\left [ I - {\mathcal{K}}'   (\varphi) \right]^{-1}  y_n 
& = & U (T (\varphi)) h^{4 r} + O \left ( h^{2 r} \max \left \{ \tilde{h}^d, h^{2 r + 2} \right \} \right ).
\end{eqnarray*}
Using the estimate (\ref {eq:*}), it then follows that 
the second   term of (\ref{eq:4.10}) is  of the order of $\tilde{h}^{d} h^{4 r}.$
We write the third term of (\ref{eq:4.10})  as
\begin{eqnarray*}
 \left [ I - {\mathcal{K}}'   (\varphi) \right]^{-1}  \left [{\mathcal{K}}_m'   (\varphi) - 
 {\mathcal{K}}'   (\varphi) \right ]  \mathcal{K}_m'   (\varphi_m) L_m (I - Q_n) \mathcal {K}_m'  (\varphi_m) ( Q_n - I) 
 (\varphi_m - \varphi + \varphi).
\end{eqnarray*}
From (\ref{eq:2.9}), for $v \in C[0, 1],$
\begin{eqnarray*}
\left \|\left ({\mathcal{K}}_m'   (\varphi) - {\mathcal{K}}'   (\varphi) \right ) \mathcal{K}_m'   (\varphi_m) v \right \|_\infty
\leq C_2 \left \|\mathcal{K}_m'   (\varphi_m) v \right \|_{d, \infty} \tilde{h}^d \leq  C_2 C_5 \left (  \sum_{i=1}^\rho  |\omega_i| \right )  \|v\|_\infty
\tilde{h}^d.
\end{eqnarray*}
It follows that
\begin{eqnarray}\label{eq:4.11}
\left \| \left ({\mathcal{K}}_m'   (\varphi) - 
 {\mathcal{K}}'   (\varphi) \right )  \mathcal{K}_m'   (\varphi_m) \right \| = O \left (\tilde{h}^d \right ),
\end{eqnarray}
whereas as in Lemma 4.1, it can be proved that
\begin{eqnarray*}
\|(I - Q_n) \mathcal {K}_m'  (\varphi_m) (I - Q_n) \varphi \|_\infty = O \left (h^{3 r} \right ).
\end{eqnarray*}
Thus the third   term of (\ref{eq:4.10}) is  of the order of $\tilde{h}^{d} h^{3 r}.$

Hence
\begin{eqnarray*}
 L_m y_n &=& L_m \mathcal{K}_m'   (\varphi_m) (I - Q_n) \mathcal {K}_m'  (\varphi_m) ( Q_n \varphi_m - \varphi_m) \\
& = &
 U (T (\varphi)) h^{4 r} + O \left ( h^{2 r} \max \left \{ \tilde{h}^d, h^{2 r + 2} \right \} \right ),
\end{eqnarray*}
which proves (\ref {eq:4.8}). 
The proof of (\ref {eq:4.9}) is similar.
\end{proof}
\begin{lemma}\label{lemma:4.3}
If  $d \geq 2 r, \; \frac {\partial \kappa} {\partial u}  \in C^{ \max \{d, 3 r + 3 \}} (\Omega)
\;\;\; \mbox{and} \;\;\; f \in C^{\max \{d, 2 r + 2 \}} [0, 1],$  then
\begin{eqnarray*}
\left \|L_m  
\mathcal{K}_m'  (\varphi_m) (I - Q_n) \mathcal {K}_m^{(3)} (\varphi_m) ( Q_n \varphi_m - \varphi_m)^3 \right \|_\infty
=   O \left  (h^{4 r + 2} \right ).
\end{eqnarray*}
\end{lemma}
\begin{proof}
Let
$$w_n =  \mathcal{K}_m'  (\varphi_m) (I - Q_n) \mathcal {K}_m^{(3)} (\varphi_m) ( Q_n \varphi_m - \varphi_m)^3.
$$
Then using (\ref{eq:2.15}) and (\ref{eq:4.3}), we obtain
\begin{eqnarray}\label {eq:A3}
\|w_n\|_\infty = O \left (  h^{5 r} \right ). 
\end{eqnarray}
As in Lemma  4.2, we write
\begin{eqnarray}\nonumber
L_m w_n 
& = & \left [ I - {\mathcal{K}}'   (\varphi) \right]^{-1}  w_n \\\nonumber
&+& \left [ I - {\mathcal{K}}'   (\varphi) \right]^{-1}   \left [{\mathcal{K}}_m'   (\varphi_m) - 
 {\mathcal{K}_m}'   (\varphi) \right ]
L_m  w_n\\\label {eq:4.12}
&+& \left [ I - {\mathcal{K}}'   (\varphi) \right]^{-1}  \left [{\mathcal{K}}_m'   (\varphi) - 
 {\mathcal{K}}'   (\varphi) \right ] L_m
 w_n.
\end{eqnarray}
Note that
\begin{eqnarray*}
\left [ I - {\mathcal{K}}'   (\varphi) \right]^{-1}  w_n & = & \left [ I - {\mathcal{K}}'   (\varphi) \right]^{-1}  
\mathcal{K}'   (\varphi) (I - Q_n) \mathcal {K}_m^{(3)} (\varphi_m) ( Q_n \varphi_m - \varphi_m)^3\\
& + & \left[ I - {\mathcal{K}}'   (\varphi) \right]^{-1}  
\left( \mathcal {K}_m'  (\varphi_m) - \mathcal{K}'   (\varphi) \right) (I - Q_n) \mathcal {K}_m^{(3)} (\varphi_m) ( Q_n \varphi_m - \varphi_m)^3.
\end{eqnarray*}
Using the asymptotic expansions (\ref {eq:3.2}) and (\ref{eq:3.5}) we obtain
\begin{eqnarray*}
&& \left [ I - {\mathcal{K}}'   (\varphi) \right]^{-1}  
\mathcal{K}'   (\varphi) (I - Q_n) \mathcal {K}_m^{(3)} (\varphi_m) ( Q_n \varphi_m - \varphi_m)^3 \\
&& \hspace*{3 cm} = U (V_2 (\varphi))
h^{5 r} + O \left ( h^{2 r} \max \left \{ \tilde{h}^d, h^{3 r + 1} \right \} \right ).
\end{eqnarray*}
If $r = 1,$ then $V_2 (\varphi) = 0$ and if $r \geq 2,$ then $5 r \geq 4 r + 2.$
Hence it follows that
\begin{eqnarray}\nonumber
\left \|\left [ I - {\mathcal{K}}'   (\varphi) \right]^{-1}  
\mathcal{K}'   (\varphi) (I - Q_n) \mathcal {K}_m^{(3)} (\varphi_m) ( Q_n \varphi_m - \varphi_m)^3 \right \|_\infty
 = 
 O \left  (h^{4 r + 2} \right ).
\end{eqnarray}
On the other hand, using (\ref {eq:2.9}),  (\ref {eq:2.15}), (\ref{eq:A1}) and  (\ref {eq:4.2})  we obtain
\begin{eqnarray*}
\|\left [ I - {\mathcal{K}}'   (\varphi) \right]^{-1}  
\left ( \mathcal {K}_m'  (\varphi_m) - \mathcal{K}'   (\varphi) \right ) (I - Q_n) \mathcal {K}_m^{(3)} (\varphi_m) ( Q_n \varphi_m - \varphi_m)^3\|_\infty = O (\tilde{h}^d h^{3 r}).
\end{eqnarray*}
Sine $d \geq 2 r,$ it follows that
\begin{eqnarray}\label{eq:A4}
\left \|\left [ I - {\mathcal{K}}'   (\varphi) \right]^{-1}  w_n\right \|_\infty = O \left (h^{ 4 r + 2 }  \right ).
\end{eqnarray}
Using the estimates (\ref{eq:A1}) and (\ref{eq:A3}),  we see that 
\begin{eqnarray}\label{eq:A5}
\left \|\left [ I - {\mathcal{K}}'   (\varphi) \right]^{-1}   \left [{\mathcal{K}}_m'   (\varphi_m) - 
 {\mathcal{K}_m}'   (\varphi) \right ]
L_m  w_n \right \|_\infty = O \left (\tilde{h}^{d} h^{5 r} \right ).
\end{eqnarray}
We write the third term of (\ref{eq:4.12})  as
\begin{eqnarray*}
 \left [ I - {\mathcal{K}}'   (\varphi) \right]^{-1}  \left [{\mathcal{K}}_m'   (\varphi) - 
 {\mathcal{K}}'   (\varphi) \right ]  \mathcal{K}_m'   (\varphi_m) L_m (I - Q_n) \mathcal {K}_m^{(3)} (\varphi_m) ( Q_n \varphi_m - \varphi_m)^3.
 \end{eqnarray*}
From (\ref{eq:2.15}), (\ref{eq:4.2}) and (\ref{eq:4.11}), it follows that the above term is of the order of 
$\tilde{h}^{d} h^{4 r}.$ Thus,
\begin{eqnarray}\label{eq:A6}
\left \|\left [ I - {\mathcal{K}}'   (\varphi) \right]^{-1}  \left [{\mathcal{K}}_m'   (\varphi) - 
 {\mathcal{K}}'   (\varphi) \right ] L_m w_n \right \|_\infty = O \left (\tilde{h}^{d} h^{4 r} \right).
\end{eqnarray}
Since $ d \geq 2r,$ it follows from (\ref{eq:4.12}) - (\ref{eq:A6}) that
\begin{eqnarray*}
\left \|L_mw_n \right \|_\infty = \left \|L_m
\mathcal{K}_m'  (\varphi_m) (I - Q_n) \mathcal {K}_m^{(3)} (\varphi_m) ( Q_n \varphi_m - \varphi_m)^3 \right \|_\infty
=   O \left  (h^{4 r + 2} \right ).
\end{eqnarray*}
This completes the proof.
\end{proof}
We now obtain the asymptotic expansion for the first term of (\ref {eq:4.7}).

\begin{proposition}\label{prop:4.4}
If  $d \geq 2 r, \; \frac {\partial \kappa} {\partial u}  \in C^{ \max \{d, 3 r + 3 \}} (\Omega)
\;\;\; \mbox{and} \;\;\; f \in C^{\max \{d, 2 r + 2 \}} [0, 1],$  then
\begin{eqnarray*}
&&L_m  \mathcal{K}_m'   (\varphi_m) \left \{  \mathcal {K}_m (\varphi_m) -  \tilde{\mathcal{K}}_n^M (\varphi_m) \right \} \\
&& \hspace* { 2 cm} =
-  U \left (T (\varphi) + \frac {V_1 (\varphi)} {2} \right ) h^{4 r}+ O \left ( h^{ 2 r} \max \left \{ \tilde{h}^d, h^{2 r + 2} \right \} \right ).
\end{eqnarray*}
\end{proposition}
\begin{proof}
Note that
\begin{eqnarray*}
\mathcal {K}_m (\varphi_m) -  \tilde{\mathcal{K}}_n^M (\varphi_m) &= &(I - Q_n) \left(\mathcal {K}_m (\varphi_m) 
- {\mathcal{K}}_m (Q_n \varphi_m)\right).
\end{eqnarray*}
By Taylor's theorem,
\begin{eqnarray*}
\mathcal{K}_m (Q_n\varphi_m) - \mathcal{K}_m (\varphi_m) &=& \mathcal {K}_m'  (\varphi_m) (Q_n \varphi_m - \varphi_m )
+ \frac {\mathcal {K}_m''  (\varphi_m)} {2}  (Q_n \varphi_m - \varphi_m )^2\\
&+& \frac {\mathcal {K}_m^{(3)}  (\varphi_m)} {6}  (Q_n \varphi_m - \varphi_m )^3
+ \frac {\mathcal {K}_m^{(4)}  (\xi_m)} {24}  (Q_n \varphi_m - \varphi_m )^4,
\end{eqnarray*}
where $\xi_m \in B (\varphi, \delta).$
Hence
\begin{eqnarray*}
L_m  \mathcal{K}_m'   (\varphi_m) \left \{  \mathcal {K}_m (\varphi_m) -  \tilde{\mathcal{K}}_n^M (\varphi_m) \right \}
&=& - L_m  \mathcal{K}_m'   (\varphi_m)  (I - Q_n) \mathcal {K}_m'  (\varphi_m) (Q_n \varphi_m - \varphi_m )\\
&& - \frac {1} {2} \; L_m  \mathcal{K}_m'   (\varphi_m)  (I - Q_n)  {\mathcal {K}_m''  (\varphi_m)}  (Q_n \varphi_m - \varphi_m )^2\\
&&  - \frac {1} {6} \; L_m  \mathcal{K}_m'   (\varphi_m)  (I - Q_n) {\mathcal {K}_m^{(3)}  (\varphi_m)} (Q_n \varphi_m - \varphi_m )^3\\
&& - \frac {1} {24} \; L_m  \mathcal{K}_m'   (\varphi_m)  (I - Q_n) {\mathcal {K}_m^{(4)}  (\xi_m)}  (Q_n \varphi_m - \varphi_m )^4.
\end{eqnarray*}
Using (\ref{eq:2.15}) and (\ref{eq:4.3}) we deduce that
\begin{eqnarray*}
\left \|L_m  \mathcal{K}_m'   (\varphi_m)  (I - Q_n) {\mathcal {K}_m^{(4)}  (\xi_m)}  (Q_n \varphi_m - \varphi_m )^4 \right \|_\infty 
= O \left (h^{6 r} \right ).
\end{eqnarray*}
Using the above estimate,  Lemma 4.2 and Lemma 4.3, we obtain
\begin{eqnarray*}
&&L_m  \mathcal{K}_m'   (\varphi_m) \left \{  \mathcal {K}_m (\varphi_m) -  \tilde{\mathcal{K}}_n^M (\varphi_m) \right \}=
-  U \left (T (\varphi) + \frac {V_1 (\varphi)} {2} \right ) h^{4 r}+ O \left (h^{4 r + 2} \right ),
\end{eqnarray*}
which completes the proof.
\end{proof}
We quote the following result from Kulkarni-Rakshit \cite[Proposition 4.6]{Kul1}:
  \begin{eqnarray} \label{eq:A6}
  \left \|\tilde{\mathcal{K}}_n^M  (z_n^M) - \tilde{\mathcal{K}}_n^M  (\varphi_m) 
 -  \left (\tilde{\mathcal{K}}_n^M \right )' (\varphi_m) (z_n^M - \varphi_m) \right \|_\infty = O \left (\max \left \{ \tilde{h}^{d}, h^{3 r} \right \}^2 \right ).
  \end {eqnarray}
It follows that the second  term on the right hand side of (\ref{eq:4.7}) is of the order of
$\max \left \{ \tilde{h}^{d}, h^{3 r} \right \}^2 .$

We now want obtain the order of convergence of the third term on the right hand side of 
(\ref{eq:4.7}). For this purpose we first prove the following result.
\begin{lemma}\label{lemma:4.5}
If  $d \geq 2 r, \; \frac {\partial \kappa} {\partial u}  \in C^{ \max \{d, 3 r + 3 \}} (\Omega)
\;\;\; \mbox{and} \;\;\; f \in C^{\max \{d, 2 r + 2 \}} [0, 1],$  then
\begin{eqnarray}\nonumber
\left \|\mathcal{K}_m'   (\varphi_m) (I - Q_n) \mathcal{K}_m'   (\varphi_m) (I - Q_n)
 \left [\tilde{\mathcal{K}}_n^M  (\varphi_m) - \mathcal{K}_m  (\varphi_m) \right ] \right \|_\infty 
 = O \left (  h^{6 r}  \right ),\\\label {eq:4.13}
\end{eqnarray}
\begin{eqnarray} \label {eq:4.14}
  \left \|\mathcal{K}_m'   (\varphi_m) (I - Q_n) \mathcal{K}_m'   (\varphi_m) (I - Q_n)
 \left (\tilde{\mathcal{K}}_n^M \right )' (\varphi_m) \right \| 
 = O \left (h^{3 r}   \right ).
  \end{eqnarray}
\end{lemma}
\begin{proof}
Note that
 \begin{eqnarray*}
 \tilde{\mathcal{K}}_n^M  (\varphi_m) - \mathcal{K}_m  (\varphi_m) & = & (I - Q_n) (\mathcal{K}_m  (Q_n \varphi_m)
 - \mathcal{K}_m  (\varphi_m) ) \\
  &= & (I - Q_n) \mathcal{K}_m'  (\xi_m) (Q_n \varphi_m - \varphi_m),
 \end{eqnarray*}
 where $\xi_m \in B (\varphi, \delta).$
Hence
\begin{eqnarray}\nonumber
 &&\|\mathcal{K}_m'   (\varphi_m) (I - Q_n) \mathcal{K}_m'   (\varphi_m) (I - Q_n) \mathcal{K}_m'  (\xi_m) (Q_n \varphi_m - \varphi_m) \|_\infty \\\nonumber
 && \hspace*{2 cm} \leq \|\mathcal{K}_m'   (\varphi_m) (I - Q_n) \mathcal{K}_m'   (\varphi_m) (I - Q_n) \mathcal{K}_m'  (\xi_m) (Q_n \varphi - \varphi) \|_\infty\\\nonumber
 && \hspace*{2 cm}+ \|\mathcal{K}_m'   (\varphi_m) (I - Q_n) \mathcal{K}_m'   (\varphi_m) (I - Q_n) \mathcal{K}_m'  (\xi_m) (Q_n - I) ( \varphi_m - \varphi) \|_\infty.\\\label {eq:4.15}
 \end{eqnarray}
 Recall from (\ref{eq:4.6})  that
 \begin{eqnarray*}
  &&\|\mathcal{K}_m'   (\varphi_m) (I - Q_n) \mathcal{K}_m'   (\varphi_m) (I - Q_n) \mathcal{K}_m'  (\xi_m) (Q_n \varphi - \varphi) \|_\infty\\
  && \hspace*{2 cm} \leq (C_6)^2 \|\ell_m\|_{r, \infty} \|\ell_m\|_{3 r, \infty} 
  \| \mathcal{K}_m'  (\xi_m) (Q_n \varphi - \varphi) \|_{2 r, \infty} h^{4 r}
 \end{eqnarray*}
and it can be checked that
$$\| \mathcal{K}_m'  (\xi_m) (Q_n \varphi - \varphi) \|_{2 r, \infty} \leq C_5 C_6 \|\varphi\|_{2 r, \infty } h^{2 r}.$$
Thus,
\begin{eqnarray}\label{eq:4.16}
  &&\|\mathcal{K}_m'   (\varphi_m) (I - Q_n) \mathcal{K}_m'   (\varphi_m) (I - Q_n) \mathcal{K}_m'  (\xi_m) (Q_n \varphi - \varphi) \|_\infty
  = O \left (h^{6 r} \right ).
    \end{eqnarray}
  On the other hand, since $d \geq 2 r$ and $\tilde {h} \leq h,$ from (\ref {eq:2.8}) and (\ref{eq:4.4}), 
  \begin{eqnarray}\nonumber
 && \|\mathcal{K}_m'   (\varphi_m) (I - Q_n) \mathcal{K}_m'   (\varphi_m) (I - Q_n) \mathcal{K}_m'  (\xi_m) (Q_n - I) ( \varphi_m - \varphi) \|_\infty \\\nonumber
  & & \hspace*{2 cm} \leq \|\mathcal{K}_m'   (\varphi_m) (I - Q_n) \mathcal{K}_m'   (\varphi_m) (I - Q_n) \mathcal{K}_m'  (\xi_m)\| (1 + \|Q_n\|) \|\varphi_m - \varphi \|_\infty \\\label {eq:4.17}
 & & \hspace*{2 cm}   = O (h^{4 r} \tilde{h}^d ) = O \left ( h^{ 6 r} \right ).
 \end{eqnarray}
 The estimate (\ref{eq:4.13}) then follows from (\ref{eq:4.15}), (\ref{eq:4.16})  and (\ref{eq:4.17}).
 
 In order to prove (\ref{eq:4.14}),
 recall that
 $$\left (\tilde{\mathcal{K}}_n^M \right )' (\varphi_m)  = Q_n \mathcal{K}_m'   (\varphi_m) 
 + (I - Q_n) \mathcal{K}_m'   (Q_n \varphi_m) Q_n.$$
Hence
\begin{eqnarray*}
 && \mathcal{K}_m'   (\varphi_m) (I - Q_n) \mathcal{K}_m'   (\varphi_m) (I - Q_n)
 \left (\tilde{\mathcal{K}}_n^M \right )' (\varphi_m) \\
 && \hspace*{2 cm} = \mathcal{K}_m'   (\varphi_m) (I - Q_n) \mathcal{K}_m'   (\varphi_m) (I - Q_n)
 \left(\mathcal{K}_m'   (Q_n \varphi_m) - \mathcal{K}_m'   (\varphi_m) \right) Q_n \\
 && \hspace*{2 cm} + \mathcal{K}_m'   (\varphi_m) (I - Q_n) \mathcal{K}_m'   (\varphi_m) (I - Q_n)
  \mathcal{K}_m'   (\varphi_m)  Q_n 
\end{eqnarray*}
The  result (\ref{eq:4.14}) follows from  (\ref {eq:2.15}), (\ref {eq:A1}),   (\ref{eq:4.3}) and (\ref{eq:4.4}).  
\end{proof}

We now obtain the order of convergence of the third term in (\ref{eq:4.7}).
\begin{proposition}\label{prop:4.6}
If  $d \geq 2 r, \; \frac {\partial \kappa} {\partial u}  \in C^{ \max \{d, 3 r + 3 \}} (\Omega)
\;\;\; \mbox{and} \;\;\; f \in C^{\max \{d, 2 r + 2 \}} [0, 1],$  then
\begin{align*}
 \left \| \mathcal{K}_m'   (\varphi_m) \left \{
\left (\left (  \tilde{\mathcal{K}}_n^M \right )' (\varphi_m)  -  \mathcal{K}_m'   (\varphi_m) \right  ) (z_n^M - \varphi_m) \right \} \right \|_\infty 
 = O \left (  h^{3 r} \max \left \{ \tilde{h}^{d}, h^{3 r} \right \} \right ).
 \end{align*}
\end{proposition}
\begin{proof}
\noindent
Note that 
\begin{eqnarray*}
 \left (  \tilde{\mathcal{K}}_n^M \right )' (\varphi_m)  -  \mathcal{K}_m'   (\varphi_m) 
& = & (I - Q_n) (\mathcal{K}_m'  (Q_n \varphi_m )  -  \mathcal{K}_m'   (\varphi_m)) Q_n \\
&& -~  (I - Q_n) \mathcal{K}_m'   (\varphi_m) (I - Q_n).
\end{eqnarray*}
By Taylor's theorem,
\begin{align*}
\left (\mathcal{K}_m' (Q_n\varphi_m) - \mathcal{K}_m' (\varphi_m)\right )Q_n   (z_n^M - \varphi_m)  = {\mathcal {K}_m''  (\xi_m)}   \left(Q_n \varphi_m - \varphi_m, ~  Q_n(z_n^M - \varphi_m)\right)
\end{align*}
where $\xi_m \in B (\varphi, \delta).$
Hence from (\ref {eq:1.7}), (\ref{eq:2.15}) and (\ref{eq:4.3}), we obtain
\begin{eqnarray}\nonumber
&&\|\mathcal{K}_m'   (\varphi_m) (I - Q_n) {\mathcal {K}_m''  (\xi_m)}   \left(Q_n \varphi_m - \varphi_m, ~  Q_n(z_n^M - \varphi_m)\right)\|_\infty \\\nonumber
& & \hspace*{2 cm} \leq
\|\mathcal{K}_m'   (\varphi_m) (I - Q_n) {\mathcal {K}_m''  (\xi_m)}  \| \|Q_n \varphi_m - \varphi_m\|_\infty \| Q_n\| \|z_n^M - \varphi_m\|_\infty \\
\label {eq:4.18}
&&  \hspace*{2 cm} =  O\left (  h^{3 r} \max \left \{ \tilde{h}^{d}, h^{3 r} \right \} \right ).
\end{eqnarray}
Note that
 \begin{eqnarray}\nonumber
 &&\mathcal{K}_m'   (\varphi_m) (I - Q_n) \mathcal{K}_m'   (\varphi_m) (I - Q_n) (z_n^M - \varphi_m)\\\nonumber
 && \hspace*{0.5 cm} = \mathcal{K}_m'   (\varphi_m) (I - Q_n) \mathcal{K}_m'   (\varphi_m) (I - Q_n)\\\nonumber
&& \hspace*{2 cm} \times \;\;\; \left [\tilde{\mathcal{K}}_n^M  (z_n^M) - \tilde{\mathcal{K}}_n^M  (\varphi_m) 
 -  \left (\tilde{\mathcal{K}}_n^M \right )' (\varphi_m) (z_n^M - \varphi_m) \right ]\\\nonumber
 && \hspace*{0.7 cm} + ~ \mathcal{K}_m'   (\varphi_m) (I - Q_n) \mathcal{K}_m'   (\varphi_m) (I - Q_n)
 \left [\tilde{\mathcal{K}}_n^M  (\varphi_m) - \mathcal{K}_m  (\varphi_m) \right ]\\\nonumber
 && \hspace*{0.7 cm} + ~ \mathcal{K}_m'   (\varphi_m) (I - Q_n) \mathcal{K}_m'   (\varphi_m) (I - Q_n)
 \left [\left (\tilde{\mathcal{K}}_n^M \right )' (\varphi_m) (z_n^M - \varphi_m)\right ].
 \end{eqnarray}
   By  (\ref{eq:4.3}) and (\ref{eq:A6}), the first  term on the right hand side of the above equation  is of the order of $h^{2r} \max \left \{ \tilde{h}^{d}, h^{3 r} \right \}^2 .$
 By (\ref{eq:4.13}) of Lemma 4.5, the second term is of the order of $h^{6r}.$
 By (\ref{eq:4.14}) of Lemma 4.5 and by (\ref{eq:1.7}), the third term is of the order of
 $h^{3 r}  \max \left \{ \tilde{h}^{d}, h^{3 r} \right \}.$
 It follows that
 \begin{align*}
 \left \|\mathcal{K}_m'   (\varphi_m) (I - Q_n) \mathcal{K}_m'   (\varphi_m) (I - Q_n) (z_n^M - \varphi_m) \right \|_\infty 
 =  O \left (  h^{3 r} \max \left \{ \tilde{h}^{d}, h^{3 r} \right \} \right ).
 \end{align*}
 The required result follows from (\ref{eq:4.18}) and the above estimate.
 \end{proof}
 
 We now prove our main result.
\begin{theorem}\label{thm:4.8} 
If $\displaystyle {d \geq 2 r, \; \frac {\partial \kappa} {\partial u}  \in C^{ \max \{d, \; 3 r + 3 \}} (\Omega),
\; \;\; \mbox {and} \;\;\;f \in C^{\max \{d, \; 2 r + 2 \}} [0, 1]},$
 then
\begin{eqnarray*}
\tilde{z}_n^M - \varphi_m =  \left( U(T (\varphi)) + \frac {U(V_1 (\varphi ))} {2} \right ) h^{4 r } 
+ O \left (  h^{2 r} \max \left \{ \tilde{h}^{d}, h^{2 r + 2} \right \} \right ).
\end{eqnarray*}
\end{theorem}
\begin{proof}
Recall from (\ref{eq:4.1}) that
\begin{eqnarray}\nonumber
\tilde{z}_n^M - \varphi_m =  \mathcal {K}_m' (\varphi_m) (z_n^M - \varphi_m) 
+ O \left (\max \left \{\tilde{h}^{d}, h^{3 r} \right \}^2 \right).
\end{eqnarray}
Recall from (\ref{eq:4.7}) that
\begin{eqnarray}\nonumber
&&\mathcal{K}_m'   (\varphi_m) ( z_n^M - \varphi_m ) \\\nonumber
& = &  
 - L_m  \mathcal{K}_m'   (\varphi_m) \left \{  \mathcal {K}_m (\varphi_m) -  \tilde{\mathcal{K}}_n^M (\varphi_m) \right \}\\\nonumber
&&+ L_m  \mathcal{K}_m'   (\varphi_m) \left \{
\tilde{\mathcal{K}}_n^M (z_n^M) - \tilde{\mathcal{K}}_n^M (\varphi_m)  - \left (  \tilde{\mathcal{K}}_n^M \right )' (\varphi_m) 
(z_n^M - \varphi_m) \right \}\\\nonumber
&& + L_m  \mathcal{K}_m'   (\varphi_m) \left \{
\left (\left (  \tilde{\mathcal{K}}_n^M \right )' (\varphi_m)  -  \mathcal{K}_m'   (\varphi_m) \right  ) (z_n^M - \varphi_m) \right \}.
\end{eqnarray}
Thus, by Proposition \ref{prop:4.4}, estimate (\ref{eq:A6})  and Proposition \ref{prop:4.6}, we obtain
\begin{eqnarray*}
\mathcal{K}_m'   (\varphi_m) ( z_n^M - \varphi_m ) =  \left( U(T (\varphi)) + \frac {U(V_1 (\varphi ))} {2} \right ) h^{4 r } 
+ O \left (  h^{2 r} \max \left \{ \tilde{h}^{d}, h^{2 r + 2} \right \} \right )
\end{eqnarray*}
and the proof is complete.
\end{proof}

We can now apply one step of Richardson extrapolation and obtain approximations of $\varphi$ of higher order.
Define
$$ z_n^{E} = \frac { 2^{ 4 r } \tilde {z}_{2 n}^M -  \tilde{z}_n^M} { 2^{ 4 r } - 1}.$$
Then we have the following result.

\begin{corollary} \label {Cor:1}
Under the assumptions of Theorem \ref{thm:4.8},
\begin{equation}  \label{eq:4.22}
\| z_n^{E} - \varphi_m  \|_\infty = O \left (  h^{2 r} \max \left \{ \tilde{h}^{d}, h^{2 r + 2} \right \} \right ).
\end{equation}
\end{corollary}
If we choose $\tilde{h}$ and $d$ such that $\tilde{h}^{d} \leq h^{2 r + 2},$ then
\begin{equation}\label{eq:4.25}
\| z_n^{E} - \varphi_m  \|_\infty = O \left (  h^{4 r + 2} \right ).
\end{equation}

\setcounter{equation}{0}
\section{Numerical Results}

In this section we quote the results from Kulkarni-Nidhin \cite{Kul3} to illustrate the improvement of orders of convergence by the Richardson extrapolation obtained in (\ref{eq:4.25}).

Consider
\begin{equation}\label{eq:5.1}
 \varphi (s) - \int_0^1 \frac {d s} {s + t + \varphi (t)}   = f (s), \;\;\; 0 \leq s \leq 1,
 \end{equation}
where $f$ is so chosen that
$\displaystyle { \varphi (t) = \frac {1} { t + 1}}$
is a solution of (\ref{eq:5.1}). 

We consider the following uniform partition of $[0, 1]:$
\begin{equation}\label{eq:5.2}
0 < \frac{1}{n} <\frac{2}{n} < \cdots < \frac{n}{n}=1.
\end{equation}
For $ r = 1, 2,$ let $\mathcal{X}_n$  be the space of piecewise polynomials of degree $\leq r - 1$ 
with respect to the partition (\ref{eq:5.2}). The collocation points are chosen to be $r$ Gauss points in
each subinterval. 

If $\mathcal{X}_n$ is the space of piecewise constant functions,  then we choose the composite 2 point Gaussian quadrature
with respect to the uniform partition of $[0, 1]$ with $256$ intervals. The computations are done for $n = 2, 4, 8, 16$ and $32.$ 
Thus, $$ r = 1, \; d = 4, \; \tilde{h} = 2^{-8}, \; h \geq 2^{-5}  \;\;\;
\mbox{and hence} \;\;\;  \tilde{h}^d = 2^{- 32 }\leq 2^{- 20 } \leq h^{2 r + 2}$$
and the conditions of Theorem 4.7 are satisfied.

If $\mathcal{X}_n$ is the space of piecewise linear functions, then  we chose the composite 2 point Gaussian quadrature
with respect to the uniform partition of $[0, 1]$ with $n^2$ intervals. Thus,
$$r = 2, \; \tilde {h} = h^2, \; d = 4 \; \mbox{and hence} \;\;\;  \tilde{h}^d = h^8 \leq h^{6} = h^{2 r + 2}$$
and the conditions of Theorem 4.7 are satisfied.

\begin{center}
Table 4.4, {\bf Interpolation at Midpoints: $ r = 1$}
\vspace*{0.5 cm}

\begin{tabular} {|c|cc|cc|cc|}\hline
$n$ 
& $\| \varphi - z_n^M \|_\infty$ & $\delta^M$ & $\| \varphi  - \tilde{z}_n^M\|_\infty$ & $\delta^{IM}$ & $\| \varphi - z_n^{E} \|_\infty$ & $\delta^{E}$\\
\hline
2&     $ 6.92 \times 10^{-3}   $ &               & $  3.30 \times 10^{- 4}   $ &  &  & \\
4&    $ 1.02 \times 10^{- 3}   $ & $ 2.76 $ & $ 2.13 \times 10^{- 5}   $ & $ 3.95 $ & $ 7.31 \times 10^{-7}   $ &  \\
8&    $ 1.40 \times 10^{- 4}  $ & $ 2.87 $ & $ 1.34 \times 10^{-  6} $ & $ 3.99 $ & $ 6.81 \times 10^{- 9}   $ & $ 6.75 $ \\
16&    $ 1.82 \times 10^{- 5}  $ & $ 2.94 $ & $ 8.37 \times 10^{- 8} $ & $ 4.00 $ & $ 8.46 \times 10^{- 11}   $ & $ 6.33 $ \\
32&    $ 2.27 \times 10^{- 6 }   $ & $ 3.00 $ & $ 5.23 \times 10^{- 9 }$ & $ 4.00 $ & $ 1.01 \times 10^{- 12}   $ & $ 6.39 $ \\\hline
\end{tabular}
\end{center}

\vspace*{0.5 cm}

\begin{center}  Table 4.6, {\bf Interpolation at Gauss 2 points: $ r = 2$}

\vspace*{0.5 cm}

\begin{tabular} {|c|cc|cc|cc|}\hline

$n$ 
& $\| \varphi - z_n^M \|_\infty$ & $\delta^M$ & $\| \varphi  - \tilde{z}_n^M\|_\infty$ & $\delta^{IM}$ & $\| \varphi - z_n^{E} \|_\infty$ & $\delta^{E}$\\
\hline

2&     $ 5.06 \times 10^{-4}   $ &               & $ 6.47 \times 10^{-5}   $ & & & \\
4&     $ 1.07 \times 10^{-5}   $ & $ 5.56 $ & $ 2.09 \times 10^{-7}   $ & $ 8.27 $& $ 4.37 \times 10^{-8} $&   \\
8&     $ 1.85 \times 10^{- 7}  $ & $ 5.86 $ & $ 8.45 \times 10^{-10} $ & $ 7.95 $& $ 2.67 \times 10^{- 11}   $ &  $10.68$ \\
16&    $ 3.07 \times 10^{- 9}  $ & $ 5.90 $ & $ 3.35 \times 10^{- 12} $ & $ 7.98 $ & $ 4.73 \times 10^{- 14}   $ & $ 9.14 $ \\
32&  $ 4.74 \times 10^{- 11}   $ & $6.02$ & $ 1.34 \times 10^{-14}$ & $7.96$ & $ 2.11 \times 10^{- 15}   $ & $ 4.49 $    \\\hline
\end{tabular}
\end{center}

\newpage
The expected orders of convergence from (\ref{eq:1.7}), (1.7)  and (\ref{eq:4.22}) are as follows:



Discrete Modified Projection  Solution:  $\delta^M = 3 r,\;\; $ 

Discrete Iterated Modified Projection Solution: $\delta^{IM} = 4  r, \;\;\ $ 

Extrapolated Solution: $\delta^{E} =  4 r + 2.$


It is clear from the above tables that the computed orders of convergence match well with the theoretical 
orders of convergence and that the extrapolated solution improves upon the discrete iterated modified 
projection solution.

\end{document}